\documentclass[english]{amsart}

\usepackage{amsmath,amssymb,enumerate}
\usepackage{amscd,amsxtra,calc}

\usepackage[T1]{fontenc}
\usepackage[all]{xy}

\usepackage{babel}
\usepackage{amstext}
\usepackage{amsmath}
\usepackage{amsfonts}
\usepackage{latexsym}
\usepackage{ifthen}
\usepackage{verbatim}
\usepackage{enumerate}
\usepackage[stable]{footmisc}
\usepackage{xypic}
\xyoption{all}
\newcommand{\rk}{{\rm rk}}

\newtheorem{lemma1}{}[section]

\newenvironment{lemma}{\begin{lemma1}{\bf Lemma.}}{\end{lemma1}}
\newenvironment{example}{\begin{lemma1}{\bf Example.}\rm}{\end{lemma1}}

\newenvironment{theorem}{\begin{lemma1}{\bf Theorem.}}{\end{lemma1}}
\newenvironment{proposition}{\begin{lemma1}{\bf Proposition.}}{\end{lemma1}}
\newenvironment{corollary}{\begin{lemma1}{\bf Corollary.}}{\end{lemma1}}
\newenvironment{remark}{\begin{lemma1}{\bf Remark.}\rm}{\end{lemma1}}
\newenvironment{definition}{\begin{lemma1}{\bf Definition.}}{\end{lemma1}}

\newenvironment{conjecture}{\begin {lemma1}{\bf Conjecture.}}{\end{lemma1}}

\newenvironment{remarks}{\begin{lemma1}{\bf Remarks.}\rm}{\end{lemma1}}

\newenvironment{remark*}{{\bf Remark.}}{}
\newenvironment{example*}{{\bf Example.}}{}

\newcommand{\R}{\ensuremath{\mathbb{R}}}
\newcommand{\Q}{\ensuremath{\mathbb{Q}}}
\newcommand{\Z}{\ensuremath{\mathbb{Z}}}
\newcommand{\C}{\ensuremath{\mathbb{C}}}

\newcommand{\PP}{\ensuremath{\mathbb{P}}}
\newcommand{\D}{\ensuremath{\mathbb{D}}}

\newcommand{\Homsheaf} { \ensuremath{ \mathcal{H} \! om}}

\newcommand{\merom}[3]{\ensuremath{#1:#2 \dashrightarrow #3}}

\newcommand{\holom}[3]{\ensuremath{#1:#2  \rightarrow #3}}
\newcommand{\fibre}[2]{\ensuremath{#1^{-1} (#2)}}

\makeatletter
\ifnum\@ptsize=0  \fi \ifnum\@ptsize=2  \fi 

\newcommand\sF{{\mathcal F}}
\newcommand\sG{{\mathcal G}}
\newcommand\sH{{\mathcal H}}

\newcommand\sO{{\mathcal O}}

\DeclareMathOperator*{\nons}{nons}

\newcommand{\pic}[0]{\operatorname{Pic}}

\newcommand{\sing}[0]{\operatorname{sing}}

\newcommand{\NEX}{\overline{\mbox{NE}}(X)}

\usepackage{xcolor}

\usepackage{url}

\title{Fano varieties with split tangent sheaf}
\date{February 17, 2026}

\makeatletter
\@namedef{subjclassname@2020}{\textup{2020} Mathematics Subject Classification}
\makeatother

\subjclass[2020]{14J45, 14E30, 14F10, 14M22}
\keywords{Fano type variety, split tangent sheaf, foliation}

\author {Andreas H\"oring}

\address{Andreas H\"oring, Universit\'e C\^ote d'Azur, CNRS, LJAD, France}
\email{Andreas.Hoering@univ-cotedazur.fr}

\begin{document}

\begin{abstract} 
Let $X$ be a mildly singular Fano variety such that the tangent sheaf
is a direct sum. We show that the direct factors are algebraically integrable,
so the infinitesimal decomposition induces a product structure on a quasi-\'etale cover of $X$.
\end{abstract}

\maketitle

\section{Introduction}

\subsection{Main result}

We say that a normal projective variety $X$ has a split tangent sheaf if there are non-zero subsheaves $\sF_i \subset T_X$ such that
$T_X = \sF_1 \oplus \sF_2$. Varieties with this property appear naturally whenever one aims to prove a decomposition
theorem, e.g. for varieties with nef anticanonical class \cite{CH19, MW21}.
In the study of singular varieties of trivial canonical class \cite{GKP16c} and
$\Q$-Fanos admitting a K\"ahler-Einstein metric \cite{DGP24} the semistability of the tangent bundle yields decompositions $T_X \simeq \oplus_{i \in I} \sF_i$ where the direct factors are integrable. For an arbitrary Fano manifold the tangent bundle is not semistable and there are many examples of distributions that are not integrable. 
The main result of this paper shows that the existence of a splitting
always guarantees the integrability:

\begin{theorem} \label{thm:main}
Let $X$ be a normal $\Q$-factorial projective variety with klt singularities such that
$T_X=V_1 \oplus V_2$. If $X$ is of Fano type (cf. Definition \ref{defn:fano-type}), the distributions $V_i$ are integrable with algebraic leaves.
\end{theorem}

This result is new even if $X$ is a smooth Fano variety, cp. \cite[Cor.3.7]{CP02}, however
we need the greater generality of varieties of Fano type for the inductive strategy of our proof.
Applying \cite[Thm.C]{DGP24} we deduce:

\begin{corollary}
Let $X$ be a normal projective variety of Fano type such that
$T_X=V_1 \oplus V_2$. Then there exists a finite quasi-\'etale cover $\holom{\nu}{X'}{X}$ such that there is a decomposition
$$
X' \simeq X_1 \times X_2
$$
into a product of normal varieties 
and the decomposition $T_X=V_1 \oplus V_2$ lifts to the canonical decomposition
$p_{X_1}^* T_{X_1} \oplus p_{X_2}^* T_{X_2}$. 
\end{corollary}

A Fano manifold is simply connected \cite{Cam92,KMM92}, so every quasi-\'etale cover is trivial. Thus we obtain:

\begin{corollary}
Let $X$ be a Fano manifold such that $T_X=V_1 \oplus V_2$. Then 
$X$ is isomorphic to a product 
$X_1 \times X_2$
and the decomposition $T_X=V_1 \oplus V_2$ corresponds to the canonical decomposition
$p_{X_1}^* T_{X_1} \oplus p_{X_2}^* T_{X_2}$. 
\end{corollary}

\subsection{Background and strategy of proof}

A well-known conjecture of Beauville predicts that a splitting of $T_X$, i.e. an infinitesimal decomposition of $X$, should always correspond
to a product structure on the universal cover:

\begin{conjecture} \cite{Bea98}
\label{conjecturebeauville}
Let $X$ be a compact K\"ahler manifold such that $T_X = V_1 \oplus V_2$, where $V_1$ and $V_2$ are {\em integrable} subbundles.
Let \holom{\mu}{\tilde{X}}{X} be the universal covering of $X$. Then there is a decomposition
$$
\tilde{X} \simeq X_1 \times X_2
$$
and the decomposition $T_X=V_1 \oplus V_2$ lifts to the canonical decomposition
$p_{X_1}^* T_{X_1} \oplus p_{X_2}^* T_{X_2}$. 
\end{conjecture}

This conjecture has been studied by many authors, e.g. \cite{Dru00, BPT06, PT13, DPPT22}.
If $X$ is simply connected (e.g. if it is Fano) it predicts that an integrable splitting always corresponds to a product structure on $X$.
For non-uniruled manifolds the existence of a splitting
implies that the direct factors are integrable \cite[Thm.1.3]{a1} while for uniruled manifolds Beauville gave a counter-example \cite[Sect.2.2]{Bea98}. His example is based on the observation  that given an isomorphism $T_X \simeq V_1 \oplus V_2$ the embedding 
$$
V_i \hookrightarrow T_X
$$
is not necessarily unique and for a sufficiently general choice of the embedding the image is not always integrable.
We expect that such counterexamples are never rationally connected:

\begin{conjecture} \label{conj:integrable}
Let $X$ be a projective manifold such that $T_X=V_1 \oplus V_2$.
Assume that $X$ is rationally connected. Then $V_1$ and $V_2$ are integrable.
\end{conjecture}

In fact \cite[Thm.1.4]{a1} shows that the integrability of one direct factor implies
that both distributions are integrable with algebraic leaves. In a similar spirit
a conjecture of Touzet \cite[Conj.1.2]{Dru17b} claims that a regular foliation on a rationally connected manifold always has algebraic leaves.

A natural strategy for Conjecture \ref{conj:integrable} is to run a MMP $X \dashrightarrow X'$ so that we get a Mori fibre space
$$
\holom{\varphi}{X'}{Y}.
$$
We have an induced decomposition $T_{X'} = V_1' \oplus V_2'$ and we show in 
Proposition \ref{prop:base-decomposition} there is an induced decomposition
$$
T_Y = W_1 \oplus W_2.
$$
Since $Y$ is also rationally connected it seems natural to proceed by induction
on the dimension by showing an analogue of Conjecture \ref{conj:integrable} for varieties
with klt singularities. Unfortunately such an analogue can't be true as shown by
Example \ref{example-not-integrable}. This difference can be explained as follows: for a singular rationally connected variety $X$ and a quotient sheaf
$$
T_X \rightarrow Q
$$
the determinant of the dual $\det Q^*$ can be pseudoeffective, this is impossible
in the smooth case. In the case of a splitting $T_X=V_1 \oplus V_2$ the pseudoeffectivity
of $\det V_i^*$ implies the integrability of $V_{3-i}$ \cite[Thm.2]{CH24}, but we learn nothing about $V_i$ itself.

If we assume $X$ to be of Fano type, the situation is much better: by work of Peternell \cite{Pet12} and Ou \cite{Ou23} we know that $\det V_i^*$ can't be pseudoeffective.
Moreover $\Q$-factorial varieties of Fano type are Mori dream spaces so we can run a MMP
with respect to $\det V_i^*$ to construct a suitable Mori fibre space structure on a birational model.
The presence of singularities creates technical difficulties but in the end the naive strategy 
of an induction on the dimension works.

{\bf Acknowledgements.} I would like to thank Cinzia Casagrande who introduced me to varieties of Fano type.
The author was supported  by the France 2030 investment plan managed by the ANR, as part of the Initiative of Excellence of Universit\'e C\^ote d'Azur, reference ANR-15-IDEX-01. 
He was also supported by the ANR-DFG project ``Positivity on K-trivial varieties'', ANR-23-CE40-0026 and DFG Project-ID 530132094,
and the ANR-Project ``Groups in algebraic geometry'' ANR-24-CE40-3526.

\section{Notation and basic facts}

We work over $\C$, for general definitions we refer to \cite{Har77},
We use the terminology of \cite{Deb01, KM98}  for birational geometry and notions from the minimal model program.

Manifolds and varieties will always be supposed to be irreducible. 
We say that a Zariski open subset $X_0 \subset X$  is big
if the complement $X \setminus X_0$ has codimension at least two in $X$.
A fibration
on a normal variety $X$ is a surjective projective morphism $\holom{f}{X}{Y}$
such that $f_* \sO_X \simeq \sO_Y$.

\subsection{Reflexive sheaves and distributions}

Let us recall some basic facts. By definition 
a coherent sheaf $\sF$ on a normal variety $X$ is reflexive if the natural morphism
$\sF \rightarrow \sF^{**}$ is an isomorphism. If two reflexive sheaves $\sF$ and $\sG$ are isomorphic on a big open
subset of $X$, they are isomorphic on $X$ \cite[Prop.1.6]{Har80}.

\begin{definition} \label{defn:saturated}
Let $\sF$ be a reflexive sheaf on a normal variety $X$. We say that a subsheaf $\sG \subset \sF$ 
is saturated if the quotient $\sF/\sG$ is torsion-free.
\end{definition}

\begin{remarks} \label{remark:saturated}
A satured subsheaf $\sG \subset \sF$ of a reflexive sheaf $\sF$ is reflexive (this follows from \cite[Prop.1.1]{Har80}).

A torsion-free sheaf on a normal variety $X$ is locally free in codimension one, i.e. there exists
a big open subset of $X$ where the sheaf is locally free.
Applying this remark to $\sG, \sF$ and $\sG/\sF$ in the definition above we obtain that there exists a big open subset $X_0 \subset X$ over which all three sheaves are locally free and the inclusion $\sG \otimes \sO_{X_0} \rightarrow \sF \otimes \sO_{X_0}$ is an injection of vector bundles.

If $\sG \subset \sF$ is a non-saturated subsheaf of a reflexive sheaf $\sF$ the saturation
$\sG^{sat} \subset \sF$ is the kernel of the surjective morphism
$$
\sF \rightarrow \sF/\sG \rightarrow  (\sF/\sG)/\mbox{Torsion}.
$$
\end{remarks}

Given a normal variety $X$ we denote by
$$
T_X := \Omega_X^*
$$
its tangent sheaf. In particular $T_X$ is reflexive \cite[Cor.1.2]{Har80}.

\begin{definition}
Let $X$ be a normal variety. A distribution on $X$ is a saturated subsheaf $\sF \subset T_X$.

We say that the distribution is integrable if $\sF \otimes \sO_{X_{\nons}} \subset T_{X_{\nons}}$ is closed under the Lie bracket. In this case we also say that $\sF$ is a foliation.

We say that a $\sF$ foliation has algebraic leaves (or is algebraically integrable)
if for a leaf $F$ passing through a general point of $X$, the manifold $F$ is a Zariski open subset of its closure $\overline{F}^{\tiny Zar} \subset X$.
\end{definition}

\begin{definition} \label{defn:relative-tangent}
Let $X$ be a normal variety, and let $\holom{f}{X}{Y}$ be a fibration. We denote by $T_{X/Y} \subset T_X$ the foliation defined
by $f$, i.e. $T_{X/Y}$ is the saturation of the image of the morphism
$\Omega_{X/Y}^* \rightarrow \Omega_X^* \simeq T_X$.
\end{definition}

\begin{lemma} \label{lemma:bertini}
Let $X$ be a normal variety, and let $\holom{f}{X}{Y}$ be a fibration. Let $\sF$ be a reflexive sheaf on $X$. Then there exists a non-empty Zariski open subset $Y_0 \subset Y$ such that
$\sF \otimes \sO_{X_y}$ is reflexive for every $y \in Y$.
\end{lemma}

\begin{proof}
By generic flatness we can assume that $\sF$ is flat over $Y$. Now the statement follows
from \cite[Thm.12.2.1]{Gro66}.
\end{proof}

\begin{corollary} \label{cor:bertini}
Let $X$ be a normal variety, and let $\holom{f}{X}{Y}$ be a fibration. There exists a non-empty Zariski open subset $Y_0 \subset Y$ such that for every $y \in Y$ we have
$$
T_{X/Y} \otimes \sO_{X_y} \simeq T_{X_y}
$$
\end{corollary}

\begin{proof}
By Seidenberg's theorem \cite[Thm.1.7.1]{BS95}
we can find $Y_0 \subset Y$ such that for $y \in Y_0$ the fibres $X_y$ are normal and have the expected dimension.
In particular $T_{X_y}$ is reflexive. By Lemma \ref{lemma:bertini} we can also assume
that $T_{X/Y} \otimes \sO_{X_y}$ is reflexive. Since $X_y$ is normal it suffices to show that
the sheaves are isomorphic on the smooth locus of $X_y$ \cite[Prop.1.6]{Har80}. Yet $X_{y, \nons}$
is contained in the smooth locus of the morphism $f$, on this set the statement is clear \cite[III,Prop.10.4]{Har77}.
\end{proof}

\begin{lemma} \label{lemma:direct-sum}
Let $X$ be a normal variety and let $\sF \simeq \sF_1 \oplus \sF_2$ be a reflexive sheaf.
\begin{itemize}
\item Then $\sF_1$ and $\sF_2$ are reflexive.
\item Let $x \in X$ be a point such that $\sF$ is locally free in $x$. Then $\sF_1$ and $\sF_2$
are locally free in $x$ and we have a direct sum of vector spaces
$$
\sF_x \simeq \sF_{1,x} \oplus \sF_{2,x}
$$
\end{itemize}
\end{lemma}

\begin{proof}
It is clear that the subsheaves $\sF_i \subset \sF$ are torsion-free. Since $\sF_i$
is the kernel of the quotient $\sF \rightarrow \sF/\sF_{i} \simeq \sF_{3-i}$ we obtain from
Remark \ref{remark:saturated} that $\sF_i$ is reflexive.

The second statement is a consequence of Nakayama's lemma, cf. \cite[II,Ex.5.8.(c)]{Har77}.
\end{proof}

\begin{remark} \label{remark:direct-sum}
Let $X$ be a normal variety, and let $\holom{f}{X}{Y}$ be a fibration.
Let $\sG$ be locally free sheaf on $Y$ such that $f^* \sG \simeq \sF_1 \oplus \sF_2$.
By the projection formula 
we have an induced isomorphism
$$
\sG \simeq f_* f^* \sG \simeq f_* \sF_1 \oplus f_* \sF_2.
$$
\end{remark}

\begin{definition} \label{defn:intersection-sheaves}
Let $X$ be a normal variety. Let $\sF$ be a reflexive sheaf on $X$, and let
$\sG$ and $\sH$ be saturated subsheaves of $\sF$. Then the intersection $\sG \cap \sH$
is the kernel of the morphism
$$
\sG \hookrightarrow \sF \twoheadrightarrow \sF/\sH.
$$
\end{definition}

\begin{remark} \label{remark:intersection-sheaves}
It is straightforward to verify that the definition can be symmetrised, i.e. $\sG \cap \sH$
is also the kernel of the morphism
$$
\sH \hookrightarrow \sF \twoheadrightarrow \sF/\sG.
$$
Moreover $\sG \cap \sH$ is a saturated subsheaf in $\sG$ (resp. $\sH$): 
the quotient $\sG/(\sG \cap \sH)$ injects in the torsion-free sheaf $\sF/\sH$ (resp. $\sF/\sG$),
hence is also torsion-free. In particular $\sG \cap \sH$ is a reflexive sheaf.
\end{remark}

\begin{lemma} \label{lemma:extend-generic-splitting}
Let $X$ be a normal variety, and let $\sF=\sF_1 \oplus \sF_2$ be a reflexive sheaf.
Let $\sG \subset \sF$ be a saturated subsheaf.
If the inclusion 
$$
(\sF_1 \cap \sG)  \oplus 
(\sF_2 \cap \sG)  \hookrightarrow \sG
$$
is an isomorphism in the generic point of $X$, it is an isomorphism.
\end{lemma}

\begin{proof}
By Remark \ref{remark:intersection-sheaves} the sheaves
$\sF_1 \cap \sG$ and $\sF_2 \cap \sG$ are reflexive, so
it is sufficient to show that the direct sum coincides with $\sG$ on a big open subset.
Thus we can assume without loss of generality that
$$
\sF_1 \cap \sG, \ \sF_2 \cap \sG, \ \sG, \ \sF_1, \  \sF_2, \ \sF
$$
are locally free and $\sG$ is a subbundle of $\sF$. By assumption we have
$$
\rk (\sF_1 \cap \sG) + \rk (\sF_2 \cap \sG)  = \rk \sG,
$$
so it is sufficient to show that for every $x \in X$ the linear map 
$$
(\sF_1 \cap \sG)_x \oplus (\sF_2 \cap \sG)_x  \rightarrow \sG_x,
$$
is injective. Yet 
this is clear since $\sF_{1,x} \oplus \sF_{2,x}  \rightarrow \sF_x$ is injective.
\end{proof}

\subsection{First Chern class and split tangent sheaf}

Recall that on a normal variety the class group $\mbox{Cl}(X)$ is isomorphic
to the group of reflexive sheaves of rank one with group law the reflexive
tensor product $\sF \boxtimes \sG := (\sF \otimes \sG)^{**}$.
Thus if $\sF$ is a reflexive sheaf of rank one such that the corresponding Weil divisor
class is $\Q$-Cartier  we can use the map
$$
\mbox{Pic}(X) \otimes \Q \rightarrow H^1(X, \Omega_X^{[1]})
$$
to define the first Chern class $c_1(\sF)$.
Moreover we have 

\begin{lemma} \label{lemma:smoothlocus}   \cite[Lemma 6.4]{AD14}, \cite{AD14b}
Let $X$ be a normal variety. Then the restriction map
$$
H^1(X, \Omega_X^{[1]}) \rightarrow H^1(X_{\nons}, \Omega_{X_{\nons}})
$$
is injective.
\end{lemma}

\begin{lemma} \label{lemma:dependent}
Let $X$ be a normal variety such that $T_X=V_1 \oplus V_2$. Suppose that $\det V_1$ and $\det V_2$ are $\Q$-Cartier,
denote by $c_1(V_i)$ the corresponding element in $H^1(X, \Omega_X^{[1]})$. 
Then we have
$$
c_1(V_i) \in H^1(X, V_i^*) \subset H^1(X, \Omega_X^{[1]}).
$$
In particular if $\{ c_1(V_1), c_1(V_2)\}$ are linearly dependent in $H^1(X, \Omega_X^{[1]})$,
at least 
one of them is zero.
\end{lemma}

\begin{proof}
By Lemma \ref{lemma:smoothlocus} it is sufficient to show the statement in the case where $X$ is smooth. By \cite[Lemma 3.1.]{Bea98} we then know that $c_1(V_i) \in H^1(X, V_i^*)$.
This shows the first statement for the second statement 
just note that the sum $H^1(X, \Omega_X^{[1]}) = H^1(X, V_1^*) \oplus H^1(X, V_2^*)$ is direct.
\end{proof}

\begin{definition} \label{defn:fano-type}
Let $X$ be a normal projective variety. We say that $X$ has Fano type if there exists a boundary
divisor such that the pair $(X, \Delta)$ is klt and $-(K_X+\Delta)$ is ample.
\end{definition}

By the work of Peternell \cite{Pet12}, Ou \cite[Thm.1.10]{Ou23} and Jovinelly-Lehmann-Riedl \cite{JLR25} we know that the tangent sheaf of a variety of Fano type has strong positivity properties. Let us state the result in the form that we will need:

\begin{proposition} \label{prop:fano-type}
Let $(X, \Delta)$ be a variety of Fano type  and let $T_X \rightarrow Q$ be a torsion-free quotient sheaf of positive rank. Then $\det Q \cdot H^{\dim X-1}>0$ for every ample divisor $H$ on $X$.
\end{proposition}

\begin{remark*}
By definition the determinant $\det Q = (\bigwedge^{\rk Q} Q)^{**}$ is a reflexive sheaf, so it 
corresponds to a Weil divisor class.  Even if $\det Q$ is not $\Q$-Cartier the
intersection number $\det Q \cdot H^{\dim X-1}$ is well-defined by taking the intersection
of $H^{\dim X-1}$ with a codimension $1$-cycle representing $\det Q$.
\end{remark*}

\begin{proof}
Since $(X, \Delta)$ is klt there exists a small $\Q$-factorial modification $\holom{\tau}{X'}{X}$  \cite[Cor.1.37]{Kol13}. The quotient $T_X \rightarrow Q$
induces a quotient $T_{X'} \rightarrow Q'$ such that $Q'$ is isomorphic to the coherent sheaf $\tau^* Q$ on a big open subset. In particular we have 
$$
\det Q' \cdot (\tau^* H)^{\dim X-1} =
\det Q \cdot H^{\dim X-1}.
$$
By \cite[Lemma 2.8(iii)]{PS09} the variety $X'$ is of Fano type. Thus we know by 
\cite[Prop.2.13]{JLR25} that $\det Q'$ is pseudoeffective and not numerically trivial. Since $\tau^* H$ is nef and big and $\tau$ is small the Hodge index theorem  implies that $\det Q' \cdot (\tau^* H)^{\dim X-1}>0$.
\end{proof}

\begin{corollary} \label{cor:fano-type}
Let $(X, \Delta)$ be a Fano type variety 
such that $T_X = V_1 \oplus V_2$ and $\rk V_i>0$ for $i=1,2$. If $\det V_1$ and $\det V_2$ are $\Q$-Cartier
we have $\rho(X) \geq 2$.
\end{corollary}

\begin{proof}
If $\rho(X)=1$ we know by Proposition \ref{prop:fano-type} that  $\det V_1$ and $\det V_2$
are ample. For the same reason they are linearly dependent, a contradiction to Lemma \ref{lemma:dependent}.
\end{proof}

\section{Fibrations on varieties with split tangent sheaf}
\label{section:fibrations}

In this section we study the behaviour of the splitting of the tangent sheaf for varieties
admitting a fibration. Combined with general tools from MMP this will allow to show the main result in Section \ref{section:proofs}.

\subsection{Relative tangent sheaf}

The following statement is a straightforward adaptation of \cite[Prop.3.16]{a2}, we give the details of the proof for the convenience of the reader.

\begin{proposition} \label{propositionungeneric}
Let $X$ be a normal variety such that  $T_X=V_1 \oplus V_2$.
Let $\holom{f}{X}{Y}$ be a fibration such that the general fibre $F$ satisfies $H^0(F, \Omega_F^{[1]})=0$. 
Then one has
$$
T_{X/Y} = (T_{X/Y} \cap V_1) \oplus (T_{X/Y} \cap V_2).
$$
\end{proposition}

\begin{proof}
We choose a non-empty $Y_0 \subset Y$ such that $Y_0$ is smooth and all the $f$-fibres are normal \cite[Thm.1.7.1]{BS95}. We will show that the statement holds on $X_0 := \fibre{f}{Y_0}$, by Lemma \ref{lemma:extend-generic-splitting} this implies the statement on $X$.

In order to simplify the notation we will write $X$ and $Y$ rather that $X_0$ and $Y_0$.
Let us first prove that
\begin{equation} \label{nomorphism}
\mbox{Hom}(T_{X/Y}, f^* T_{Y}) = H^0(X, \Homsheaf(T_{X/Y}, f^* T_{Y})) =0.
\end{equation}
Since $\Homsheaf(T_{X/Y}, f^* T_{Y})$ is torsion-free it is sufficient to show that
for a general fibre $F$ 
$$
H^0(F, \Homsheaf(T_{X/Y}, f^* T_{Y}) \otimes \sO_F) = 0. 
$$
Since $f^* T_{Y}$ is trivial on $F$ and 
$T_{X/Y} \otimes \sO_F \simeq T_F$ by Corollary \ref{cor:bertini}
this follows from
$$
H^0(F, \Homsheaf(T_F, \sO_F)) \simeq H^0(F, \Omega_F^{[1]})=0.
$$

Let 
$$
\holom{i_1}{V_1}{T_X}, \qquad  \holom{i_2}{V_2}{T_X}
$$
be the inclusion maps, and
let 
$$
\holom{q_1}{T_X}{V_1}, \qquad \holom{q_2}{T_X}{V_2}
$$
be the projections given by the decomposition $T_X = V_1 \oplus V_2$.  Furthermore we have
the natural inclusion \holom{i}{T_{X/Y}}{T_X} and a generically surjective map 
\holom{q}{T_X}{f^* T_Y}.
Taking the composition of these maps we obtain the following (typically not
commutative) diagram:
\[
\xymatrix{
& & V_1 \ar @/^1pc/[d]^{i_1} \ar @/^1pc/[rd]^\beta & &
\\
0 \ar[r] & T_{X/Y}  \ar[r]^{i} \ar @/^1pc/[ru]^\alpha \ar @/_1pc/[rd]_\gamma & T_X \ar[r]^q
\ar @/^1pc/[u]^{q_1} \ar @/_1pc/[d]_{q_2} & f^* T_Y & 
\\
& & V_2 \ar @/_1pc/[u]_{i_2} \ar @/_1pc/[ru]_\delta & &
}
\]
Let $x \in X$ be a point where $f$ is smooth, then $X$ is smooth in $x$,
the reflexive sheaves $V_1, V_2$ are locally free in $x$ (cf. Lemma \ref{lemma:direct-sum}) and
$T_{X/Y} \subset T_{X}$ is a subbundle near $x$.
Moreover we have
$$
\ker \alpha_x = T_{X/Y,x} \cap V_{2,x}, \qquad \ker \gamma_x= T_{X/Y,x} \cap V_{1,x}.
$$
Since $(T_{X/Y,x} \cap V_{2,x}) \cap (T_{X/Y,x} \cap V_{1,x})=0$ this implies
\begin{equation}
\label{eqn:star}
\rk \alpha_x + \rk \gamma_x = \dim T_{X/Y,x} -
\dim \ker \alpha_x + \dim T_{X/Y,x} - \dim \ker \beta_x
\geq \rk T_{X/Y,x}.
\end{equation}
In particular $\rk \alpha + \rk \gamma \geq \rk T_{X/Y}$ and we claim that equality holds.
Granting this for the time being, we show how to conclude:
since $\ker \alpha = T_{X/Y} \cap V_{2}$ and $\ker \gamma= T_{X/Y} \cap V_{1}$,
we see that $\rk (T_{X/Y} \cap V_{1}) + \rk (T_{X/Y} \cap V_{2}) = \rk T_{X/Y}$
and clearly $(T_{X/Y} \cap V_{1}) \cap (T_{X/Y} \cap V_{2})=0$. Thus the map
$$
(T_{X/Y} \cap V_{1}) \oplus (T_{X/Y} \cap V_{2}) \rightarrow  T_{X/Y}
$$
is an isomorphism in the generic point of $X$, hence everywhere by Lemma \ref{lemma:extend-generic-splitting}.

{\it Proof of the claim.}
By \eqref{nomorphism} the morphisms 
$$
\holom{\beta \circ \alpha}{T_{X/Y}}{f^* T_Y}, \qquad
\holom{\delta \circ \gamma}{T_{X/Y}}{f^* T_Y}
$$
are zero, in particular
$$
\rk \alpha \leq \rk (\ker \beta) = \rk V_1 - \rk \beta, \qquad
\rk \gamma \leq \rk (\ker \delta) = \rk V_2 - \rk \delta.
$$
The morphism $\beta$ is the restriction of $q$ to $V_1$ and $\delta$ is the restriction of $q$ to $V_2$.
Since $q$ is generically surjective this implies that 
\[
\rk \beta + \rk \delta \geq \rk f^* T_Y = \rk T_X - \rk T_{X/Y}.
\]
Putting these inequalities together  with \eqref{eqn:star} we obtain
\begin{eqnarray*}
\rk T_{X/Y} & \leq & \rk \alpha + \rk \gamma
\\
& \leq & \rk V_1 + \rk V_2 - \rk \beta - \rk \delta
\\
& \leq & \rk T_X - (\rk T_X- \rk T_{X/Y_0}) = \rk T_{X/Y}.
\end{eqnarray*}
Thus all the inequalities are equalities.
\end{proof}

\begin{proposition} \label{prop:inclusion-rel-tangent}
Let $X$ be a normal $\Q$-factorial projective variety with klt singularities
such that $T_X=V_1 \oplus V_2$.
Let $\holom{f}{X}{Y}$ be an elementary $K_X$-negative contraction of fibre type
such that $\det V_i$ is $f$-ample. Then we have 
$$
T_{X/Y} \subset V_i.
$$
\end{proposition}

\begin{proof}
In order to simplify the notation assume that $\det V_1$ is $f$-ample.
Since $X$ is klt and the contraction $f$ is $K_X$-negative
a general $f$-fibre $F$ is a Fano variety with klt singularities. Therefore 
$F$ is rationally connected by \cite[Cor.1.5]{HM07} and
$H^0(F, \Omega_F^{[1]})=0$ by \cite[Thm.5.1]{GKKP11}.
By Proposition \ref{propositionungeneric} this implies
\begin{equation}
\label{split-relative}
T_{X/Y} =  (T_{X/Y} \cap V_1) \oplus (T_{X/Y} \cap V_2).
\end{equation}
Set $\sF_i := T_{X/Y} \cap V_i$, then $\sF_i$ is a reflexive sheaf by Lemma \ref{lemma:direct-sum}
and $\det \sF_i$ is $\Q$-Cartier since $X$ is $\Q$-factorial.
By Corollary \ref{cor:bertini}
we obtain a decomposition
$$
T_F = T_{X/Y} \otimes \sO_F \simeq (\sF_1 \otimes \sO_F) \oplus (\sF_2 \otimes \sO_F)
$$
and $\det(\sF_i \otimes \sO_F) \simeq (\det \sF_i) \otimes \sO_F$ is $\Q$-Cartier.
Since $f$ is an elementary Mori contraction we know that the restriction map
$$
\pic(X) \otimes \Q \rightarrow \pic(F) \otimes \Q
$$
has rank one. Therefore $c_1(\sF_1 \otimes \sO_F)$ and $c_1(\sF_2 \otimes \sO_F)$
are linearly dependent. By Lemma \ref{lemma:dependent} this implies that
$c_1(\sF_1 \otimes \sO_F)=0$ or $c_1(\sF_2 \otimes \sO_F)=0$. Yet $(F, \Delta_F)$ is log Fano
so Proposition \ref{prop:fano-type} yields $\sF_1 \otimes \sO_F$ or $\sF_2 \otimes \sO_F$ has rank zero.
By \eqref{split-relative} this implies
$T_{X/Y} \subset V_2$ or  $T_{X/Y} \subset V_1$.

Arguing by contradiction  assume that $T_{X/Y} \subset V_2$: then
$V_1^* \otimes \sO_F$ is contained in the kernel of the map
$$
\Omega_X^{[1]} \otimes \sO_F = (V_1^* \otimes \sO_F) \oplus (V_2^* \otimes \sO_F) 
\rightarrow \Omega_F^{[1]}.
$$
By Lemma \ref{lemma:dependent} 
we have $c_1(V_1) \in H^1(X, V_1^*)$, so this shows
that $c_1(V_1)|_F \in H^1(F, \Omega_F^{[1]})$ is zero. Yet by assumption $\det V_1$
is $f$-ample, a contradiction.
\end{proof}

\subsection{Descent of the splitting}

Let $X$ be a normal variety, and let $\holom{f}{X}{Y}$ be an equidimensional fibration onto a manifold $Y$. 
Dualising the natural map $\holom{df}{f^* \Omega_Y}{\Omega_X}$ 
we obtain the tangent map 
$$
T_f: T_X \rightarrow f^* T_Y.
$$
Given a distribution $\sF \subset T_X$, we denote by
$$
(T_f(\sF))^{sat} \subset f^* T_Y
$$
the saturation of the subsheaf $T_f(\sF) \subset f^* T_Y$. Therefore we have an exact
sequence
$$
0 \rightarrow (T_f(\sF))^{sat} \rightarrow f^* T_Y \rightarrow Q \rightarrow 0
$$
with $Q$ a torsion-free sheaf. Pushing forward we obtain an exact sequence
$$
0 \rightarrow f_*\left( (T_f(\sF))^{sat}\right) \rightarrow f_* f^* T_Y \simeq T_Y
\rightarrow f_* Q.
$$
Since $f_* Q$ is torsion-free, the quotient $T_Y/f_*\left( (T_f(\sF))^{sat}\right) \subset
f_* Q$ is torsion-free. Thus $f_*\left( (T_f(\sF))^{sat}\right) \subset T_Y$ is saturated, i.e.
a distribution on $Y$.

\begin{lemma} \label{lemma:base-decomposition}
Let $X$ be a normal variety such that $T_X = V_1 \oplus V_2$. Let 
$\holom{f}{X}{Y}$ be an equidimensional fibration  onto a manifold $Y$. Assume that $f$ has  generically reduced fibres in codimension one, i.e. for every prime divisor $B 
\subset Y$ the pull-back $f^*B$ is reduced. If $T_{X/Y} \subset V_i$ for some $i \in \{1, 2\}$, the tangent map induces a decomposition
$$
T_Y \simeq 
f_* ((T_f(V_1))^{sat}) \oplus f_* ((T_f(V_2))^{sat}).
$$
\end{lemma}

\begin{proof}
In order to simplify the notation, let us assume that $T_{X/Y} \subset V_1$. 
Since $f$ is equidimensional with generically reduced fibres
the locus $Z \subset X$ where $T_f$ is not surjective has codimension at least two in $X$.

The map $T_f$ is a surjective morphism of vector bundles on $X_{\nons} \setminus Z$, so the relative tangent sheaf $T_{X/Y}$ 
is locally free and coincides with the kernel of $T_f$ on $X_{\nons} \setminus Z$.
Therefore the inclusion $T_{X/Y} \subset V_1$ yields an isomorphism
$$
\left( T_f(V_1) \oplus T_f(V_2) \right) \otimes \sO_{X_{\nons} \setminus Z}
\simeq
f^* T_Y \otimes \sO_{X_{\nons} \setminus Z}.
$$
In particular $T_f(V_i) \subset f^* T_Y$ is a subbundle in codimension one
and the saturation $(T_f(V_i))^{sat} \subset f^* T_Y$ is isomorphic
to $(T_f(V_i))^{**}$.
Since $X_{\sing} \cup Z$ has codimension at least two in $X$ and $f^* T_Y$ is locally free (hence reflexive), we can extend the decomposition by taking biduals:
$$
(T_f(V_1))^{sat} \oplus (T_f(V_2))^{sat}
\simeq f^* T_Y.
$$
By Remark \ref{remark:direct-sum} this induces the desired decomposition on $T_Y$.
\end{proof}

Our next goal is to extend Lemma \ref{lemma:base-decomposition} to certain fibrations with
non-reduced fibres, this requires to understand the behaviour of the splitting
under finite covers (cf. \cite[Sect.5.2]{Dru21} for similar statements for weakly regular foliations).

Let $\holom{\eta}{U'}{U}$ be a finite map between complex manifolds.
Let $W \subset T_U$ be a subbundle of rank $n-r$, then we have an exact sequence
$$
0 \rightarrow Q \rightarrow \Omega_U \rightarrow W^* \rightarrow 0
$$
and we denote by 
$Q^{sat} \subset \Omega_{U'}$
the saturation of the image of
$
\eta^* Q \hookrightarrow \eta^* \Omega_U \hookrightarrow \Omega_{U'}.
$
The exact sequence
$$
0 \rightarrow Q^{sat} \rightarrow  \Omega_{U'} \rightarrow S \rightarrow 0,
$$
defines an injective morphism 
$$
\eta^{-1} W:=S^* \hookrightarrow T_{U'}.
$$

\begin{definition} In the situation above, 
we call $\eta^{-1} W \subset T_{U'}$ the lifting of the distribution $W$ to $U'$.
\end{definition}

Consider the special case where we have a holomorphic map
\begin{equation}
\label{define-eta}
\holom{\eta}{U' \simeq \D^n}{U \simeq \D^n},
(z_1, \ldots, z_n) \ \mapsto \ (z_1^k, z_2, \ldots, z_n) = (u_1, \ldots, u_n)
\end{equation}
where $k \geq 2, n \geq 2$ positive integers. 
Denote by 
$B = \{ u_1=0 \} \subset U$ the branch divisor and by $R = \{ z_1=0 \} \subset U'$ the (reduction of the) ramification divisor
of $\eta$.
Note that $\eta|_R: R \rightarrow B$ is an isomorphism.

The next lemma is elementary and well-known to experts, cp. \cite[Lemma 3.4]{Dru21}. Since op.cit. is stated only for foliations we give the computation for completeness sake:

\begin{lemma} \label{lemmalocalcomputation}
In the situation above for a general point $z \in R$ we have
\begin{equation} \label{restricttoramification}
(\eta^{-1} W)_z \cap T_{R,z} = W_{\eta(z)} \cap T_{B,\eta(z)},
\end{equation}
where we identify $T_{R,z}$ and $T_{B,\eta(z)}$ via the isomorphism $\eta|_R$.

Moreover the subsheaf $\eta^* Q  \hookrightarrow \Omega_{U'}$ is saturated, and hence isomorphic
to $Q^{sat}$, if and only if for a general point $u \in B$ we have
$W_u \not\subset T_{B, u}$.
\end{lemma}

\begin{proof}
Let $z \in R$ be an arbitrary point, and let $s_1, \ldots, s_r$ local generators of $Q$ near $\eta(z)$. We can write
$$
s_j = \alpha_j(u_1, \ldots, u_n) d u_1 + \sum_{i=2}^n \beta_{j,i}(u_1, \ldots, u_n) d u_i.
$$
By construction the subbundle $W$ consists in every point $u \in U$ of those $v \in T_{U,u}$ such that
$s_j(v)=0$ for all $j \in \{1, \ldots, k\}$. Denote by $\bar s_j$ the image of $s_j$ under the quotient map
$\Omega_U \otimes \sO_B \rightarrow \Omega_B.$
The kernel of this map is generated by $du_1|_B$, so in local coordinates
$$
\bar s_j = \sum_{i=2}^n \beta_{j,i}(0, u_2, \ldots, u_n) d u_i.
$$
For every point $u \in B$, the intersection $W_{u} \cap T_{B,u}$ consists of those $v \in T_{B,u}$ such that
$\bar s_j(v)=0$ for all $j \in \{1, \ldots, r\}$.
Observe that the family of holomorphic $1$-forms $\bar s_1, \ldots, \bar s_r$ has rank at least $r-1$ near $\eta(z)$ and equality
holds if and only if, up to a linear change of local generators, one has $\bar s_1=0$.
Thus we have
$W_u \subset T_{B, u}$ for every $u$ in a neighbourhood of $\eta(z) \in B$ if and only if
\begin{equation}
\label{vanish-beta}
\beta_{1,i}(0, u_2, \ldots, u_n) \equiv 0
\end{equation}
for $i=2, \ldots, n$.

The subsheaf 
$\eta^* Q \subset \Omega_{U'}$ is generated near $z$ by the pull-backs $\eta^* s_1, \ldots, \eta^* s_r$ which in local coordinates are given by
$$
\eta^* s_j = \alpha_j(z_1^k, \ldots, z_n) \cdot k z_1^{k-1} d z_1 + \sum_{i=2}^n \beta_{j,i}(z_1^k, \ldots, z_n) d z_i.
$$
Denote by $\overline{\eta^* s_j}$ the image of $\eta^* s_j$ under the quotient map
$\Omega_{U'}\otimes \sO_R \rightarrow \Omega_R$.
The kernel of this map is generated by $dz_1|_R$, so in local coordinates
$$
\overline{\eta^* s_j} = \sum_{i=2}^n \beta_{j,i}(0, z_2, \ldots, z_n) d z_i.
$$
Thus $\overline{\eta^* s_j}$ identifies to the pull-back of $\bar s_j$ under the isomorphism $\eta|_R$.

{\em 1st case: assume $W_u \not\subset T_{B, u}$ for $u \in B$ general.}
By the above the family $\bar s_1, \ldots, \bar s_r$ has rank $r$ in the generic point of $B$
and therefore the family $\overline{\eta^* s_1}, \ldots, \overline{\eta^* s_r}$ has rank $r$
in the generic point of $R$. Thus the family
$\eta^* s_1, \ldots, \eta^* s_r$ has rank $r$ in a general point of $R$ and $\eta^* Q \subset \Omega_{U'}$ is saturated.  

The intersection $(\eta^{-1} W)_{z} \cap T_{R,z}$ consists of those $v \in T_{R,z}$ such that
$\overline{\eta^* s_j}(v)=0$ for all $j \in \{1, \ldots, r\}$.
Since $\overline{\eta^* s_j}=(\eta|_R)^* \bar s_j$ the equality \eqref{restricttoramification} follows.

{\em 2nd case: assume $W_u \subset T_{B, u}$ for $u \in B$ general.}
As explained above, we can assume $\bar s_1=0$, so \eqref{vanish-beta} holds.
Therefore the $1$-form
$$
\eta^* s_1 = \alpha_1(z_1^k, \ldots, z_n) \cdot k z_1^{k-1} d z_1 + \sum_{i=2}^n \beta_{1,i}(z_1^k, \ldots, z_n) d z_i.
$$
vanishes along $R$ and $\eta^* Q \subset \Omega_{U'}$ is not saturated.
By \eqref{vanish-beta} the meromorphic $1$-forms
$
\frac{1}{z_1^{k-1}} \cdot \beta_{1,i}(z_1^k, \ldots, z_n)
$
are holomorphic, so 
$$
\frac{1}{z_1^{k-1}} \eta^* s_1, \eta^* s_2, \ldots, 
\eta^* s_r
$$
is a family of holomorphic $1$-forms that generates
$\eta^* Q \subset \Omega_{U'}$ on $U' \setminus R$ and is linearly independent
in a general point of $R$. 
Thus this family generates the saturation $Q^{sat} \subset \Omega_{U'}$.

Note also that $\overline{\frac{1}{z_1^{k-1}} \eta^* s_1} = \overline{\eta^* s_1}=0$,
so \eqref{restricttoramification} follows by the same argument as in the first case.
\end{proof}

\begin{lemma} \label{lemma:descendsplittingfinite}
Let $\holom{\eta}{U'}{U}$ be a finite map between complex manifolds. Let $W_i \subset T_U$ be distributions, and let
$\fibre{\eta}{W_i}$ be their lifting to $U'$. If we have
$
T_{U'} = \fibre{\eta}{W_1} \oplus \fibre{\eta}{W_2},
$
then 
$
T_{U} = W_1 \oplus W_2.
$
\end{lemma}

\begin{proof}
The inclusions $W_i \hookrightarrow T_U$ determine a morphism
$W_1 \oplus W_2 \rightarrow T_U$,
our goal is to show that it is an isomorphism.
The statement is trivial in the complement of the branch divisor $B$
and the sheaves $W_i$ and $T_U$ are reflexive, so it is sufficient to show the statement in a general point $u \in B$.  
This property being local, we can replace
$U$ and $U'$ by polydiscs and assume that $\eta$ is given by \eqref{define-eta}.
The point $u \in B$ being general and the subsheaves $W_i \subset T_U$ being saturated,
we can assume that they are subbundles.

{\em Arguing by contradiction we assume that for a general point $u \in B$
we have $W_{1,u} \not\subset T_{B,u}$ and
$W_{2,u} \not\subset T_{B,u}$.} 
By Lemma \ref{lemmalocalcomputation} the subsheaves $\eta^* Q_i \subset \Omega_{U'}$ are saturated, in particular their restrictions
$(\eta^* Q_i)\otimes \sO_R$ are contained in the image of 
$$
(\eta^* \Omega_U)\otimes \sO_R \subset \Omega_{U'}\otimes \sO_R
$$
which has rank $\dim U-1$. Since $\rk Q_1+\rk Q_2= \dim U$, this shows that 
$$
(\eta^* Q_1 \otimes \sO_R) \cap (\eta^* Q_2\otimes \sO_R) \neq 0.
$$
But then their kernels $\fibre{\eta}{W_1}$ and $\fibre{\eta}{W_2}$ intersect non-trivially, a contradiction.

{\em Assume now that, up to renumbering, we have $W_{1,u} \subset T_{B,u}$ for a general point $u \in B$.}
Since $W_{1,u} \subset T_{B,u}$ we have
$$
W_{1,u} \cap W_{2,u} = W_{1,u} \cap W_{2,u} \cap T_{B,u} 
= (W_{1,u} \cap T_{B,u}) \cap (W_{2,u} \cap T_{B,u}).
$$
By \eqref{restricttoramification} we have
$$
(W_{1,u} \cap T_{B,u}) \cap (W_{2,u} \cap T_{B,u})
=
\left((\eta^{-1} W_1)_z \cap T_{R,z}\right) \cap 
\left((\eta^{-1} W_2)_z \cap T_{R,z}\right). 
$$
Yet this intersection is zero since $T_{U'} = \fibre{\eta}{W_1} \oplus \fibre{\eta}{W_2}$.
Thus $W_{1,u} \cap W_{2,u} = 0$ and the morphism $W_1 \oplus W_2 \rightarrow T_U$
is an isomorphism in $u$.
\end{proof}

\begin{proposition} \label{prop:base-decomposition}
Let $X$ be a normal variety such that $T_X = V_1 \oplus V_2$. Let 
$\holom{f}{X}{Y}$ be an equidimensional fibration onto a manifold $Y$
such that for every prime divisor $B \subset Y$ we have $f^* B = k \sum_i D_i$
where $D_i \subset X$ are prime divisors.
If $T_{X/Y} \subset V_i$ for some $i \in \{1, 2\}$, the tangent map induces a decomposition
$$
T_Y \simeq 
f_* ((T_f(V_1))^{sat}) \oplus f_* ((T_f(V_2))^{sat}).
$$
\end{proposition}

\begin{proof}
Let $Y_0 \subset Y$ be the locus over which $f$ has reduced fibres. By Lemma \ref{lemma:base-decomposition} we have decomposition
\begin{equation}
\label{iso-open}
T_{Y_0} \simeq 
\left(f_* ((T_f(V_1))^{sat}) \otimes \sO_{Y_0}\right) \oplus \left(f_* ((T_f(V_2))^{sat}) \otimes \sO_{Y_0}\right),
\end{equation}
so we only have to show that the map
$$
f_* ((T_f(V_1))^{sat}) \oplus f_* ((T_f(V_2))^{sat}) \rightarrow T_Y.
$$
is an isomorphism in codimension one, i.e. we only have to show that the isomorphism \eqref{iso-open} extends
as an {\em isomorphism} in codimension one.
This is a local problem, so for a general point $y \in Y \setminus Y_0$ choose a polydisc $U \subset Y$ such that $B:= Y \setminus Y_0$ is the hyperplane $y_1=0$. By assumption
we have $f^* B = k \sum_i D_i$, so the multiplicity $k$ does not depend on the divisor $D_i$.
Let $\holom{\eta}{U'}{U}$ be the morphism \eqref{define-eta}, and let
$X'$ be the normalisation of the fibre product $X \times_U U'$.
Then a local computation shows that the morphism
$$
\holom{\nu}{X'}{X}
$$
is quasi-\'etale and the induced fibration
$$
\holom{f'}{X'}{U'}
$$
has generically reduced fibres in codimension one. Since $\nu$ is quasi-\'etale we have
a decomposition
$$
T_{X'} \simeq \nu^{[*]} T_X = \nu^{[*]} V_1 \oplus \nu^{[*]} V_2.
$$
By Lemma \ref{lemma:base-decomposition} we have an induced decomposition
$$
T_{U'} = f'_* ((T_{f'}(\nu^{[*]} V_1))^{sat}) \oplus f'_* ((T_{f'}(\nu^{[*]} V_2))^{sat}). 
$$
On $U' \setminus \fibre{\eta}{B}$ this decomposition coincides with the pull-back 
of the decomposition \eqref{iso-open} by $\eta$, in particular
$f'_* ((T_{f'}(\nu^{[*]} V_i))^{sat})$ is the lifting of the distribution
$W_i := f_* ((T_f(V_i))^{sat})$ to $U$.
Conclude with Lemma \ref{lemma:descendsplittingfinite}.
\end{proof}

\begin{proposition} \label{prop:descend-MFS}
Let $X$ be a normal $\Q$-factorial projective variety with klt singularities
such that $T_X=V_1 \oplus V_2$.
Let $\holom{f}{X}{Y}$ be an elementary $K_X$-negative contraction of fibre type
such that $\det V_i$ is $f$-ample.
Let $Y_0 \subset Y_{\nons}$ be the locus where $f$ is equidimensional,
and let $\holom{f}{X_0}{Y_0}$ be the restriction of $f$ over this set.
Then we have a
decomposition $T_Y \simeq W_1 \oplus W_2$ extending the decomposition
$$
T_{Y_0} \simeq 
f_* ((T_f(V_1 \otimes \sO_{X_0}))^{sat}) \oplus f_* ((T_f(V_2 \otimes \sO_{X_0}))^{sat}).
$$
Moreover $V_i$ is integrable (with algebraic leaves) if $W_i$ is integrable (with algebraic leaves).
\end{proposition}

\begin{remark*}
It is not difficult to see that $\det V_i$ $f$-ample implies that
$\det V_{3-i}$ is $f$-trivial. So while $f$ may allow us to prove
the integrability of $V_i$, it doesn't give any information about $V_{3-i}$.
\end{remark*}

\begin{proof}
In order to simplify the notation assume that $\det V_1$ is $f$-ample. By Proposition 
\ref{prop:inclusion-rel-tangent} we have $T_{X/Y} \subset V_1$.
Since $f$ is an elementary contraction we know that it does not contract a divisor and
for every prime divisor $B \subset Y$ the support of the pull-back $f^* B$ is a prime divisor. In particular $f$ satisfies the conditions of Proposition \ref{prop:base-decomposition} and we obtain a decomposition 
$$
T_{Y_0} \simeq 
f_* ((T_f(V_1 \otimes \sO_{X_0}))^{sat}) \oplus f_* ((T_f(V_2 \otimes \sO_{X_0}))^{sat}).
$$
Let $\holom{j}{Y_0}{Y}$ be the inclusion.
Since $Y_0 \subset Y$ is big 
the sheaf
$$
W_i := j_* f_* ((T_f(V_i \otimes \sO_{X_0}))^{sat}) 
$$
is a reflexive sheaf on $Y$ \cite[Prop.1.6]{Har80}. The tangent sheaf $T_Y$ being reflexive the decomposition on $Y_0$
extends to a decomposition  $T_Y \simeq W_1 \oplus W_2$.

If $W_1 \subset T_Y$ is integrable, the distribution $V_1 \subset T_X$
is the pull-back (in the sense of \cite[Sect.3.2]{Dru21}) of the foliation $W_1$ to $X$, in particular $V_1$ is integrable. Moreover a general leaf of $V_1$ coincides in its generic point with the preimage of a general leaf of $W_1$. Thus if $W_1$ has algebraic leaves, so has $V_1$.
\end{proof}

\section{Proofs of the main results}
\label{section:proofs}

We start with an introductory remark: 

\begin{remark} \label{remark:MMP-induced splitting}
Let $X$ be a normal projective variety, and let
$$
\merom{\psi}{X}{X'}
$$
be a birational map to a normal projective variety $X'$ such that $\psi$ is the composition
of birational morphisms (e.g. divisorial contractions) and rational maps that are isomorphisms
in codimension one (e.g. flips and flops). Then there exists a Zariski open subset $X_0 \subset X$
and a {\em big} open subset $X'_0 \subset X'$ such that we have an isomorphism
$X_0 \simeq X'_0$. Let $\holom{j}{X'_0}{X'}$ be the inclusion, then a decomposition $T_X = V_1 \oplus V_2$ induces a decomposition
$$
T_{X'} = V_1' \oplus V_2' := j_* (V_1 \otimes \sO_{X_0}) \oplus j_* (V_2 \otimes \sO_{X_0}).
$$
Moreover $V_i' \subset T_{X'}$ is integrable (with algebraic leaves) if and only if this holds
for $V_i \subset T_X$.
\end{remark}

\begin{proof} [Proof of Theorem \ref{thm:main}]
Since $X$ is klt there exists a small modification $\holom{\mu}{X'}{X}$ that is $\Q$-factorial \cite[Cor.1.37]{Kol13} and $X'$ is of Fano type \cite[Lemma 2.8(iii)]{PS09}.
Since $\mu$ is small we can use Remark \ref{remark:MMP-induced splitting} to see that it suffices to prove the statement for $X'$. In order to simplify the notation we will assume without loss of generality that $X$ is $\Q$-factorial. Since $X$ is of Fano type we know by \cite[Cor.1.3.2]{BCHM10} that $X$ is a Mori dream space. 

We will prove the statement by induction on the dimension, the first case being 
a klt surface such that $T_X= L_1 \oplus L_2$ where $L_i$ is a reflexive sheaf of rank one.
In this case the integrability of $L_i$ is clear, moreover the canonical sheaf $K_{L_i} = L_i^*$ is not pseudoeffective by Proposition \ref{prop:fano-type}. Thus the leaves
are algebraic \cite[Cor.II.4.2]{McQ08}.

Fix now $i \in \{1, 2\}$, our goal is to show that $V_i \subset T_X$ is integrable.
By Proposition \ref{prop:fano-type} the determinant $\det V_i^*$ is not pseudoeffective. 
Since $X$ is a Mori dream space we can run a terminating $\det V_i^*$-MMP \cite[Prop.1.11(1)]{HK00}
$X \dashrightarrow X'$ that ends with a Mori fibre space
$$
\holom{f}{X'}{Y}
$$
such that $\det V_i$ is $f$-ample. 
By Remark \ref{remark:MMP-induced splitting} the splitting $T_X=V_1 \oplus V_2$
induces a splitting $T_{X'}=V_1' \oplus V_2'$.
The MMP $X \dashrightarrow X'$ is a composition
of divisorial contraction and flips, so by \cite[Lemma 2.8(ii)]{PS09}
we know that $X'$ is of Fano type. Applying \cite[Lemma 2.8(i)]{PS09} to $f$ we also obtain
that $Y$ is Fano type.

Since $X'$ is of Fano type the extremal ray $R \subset \NEX$ contracted by $f$ is $K_{X'}$-negative, i.e. $f$ is also a Mori contraction in the classical sense. Thus we can apply
Proposition \ref{prop:inclusion-rel-tangent} to obtain $T_{X'/Y} \subset V_i'$.
If the inclusion is an equality we have shown the algebraic integrability of $V_i'$.

Assume now that the inclusion is strict.
Then Proposition \ref{prop:base-decomposition} yields a non-trivial
decomposition $T_Y = W_1 \oplus W_2$. Since $\dim X' > \dim Y$ the induction hypothesis
applies to $Y$ and yields the algebraic integrability of $W_1$ and $W_2$. 
In particular we know by Proposition \ref{prop:inclusion-rel-tangent} that
$V_i' = \fibre{f}{W_i}$ is algebraically integrable.
\end{proof}

We have a weak version of Conjecture \ref{conj:integrable} for singular varieties:

\begin{proposition}
\label{prop:weak}
Let $X$ be a normal projective variety with klt singularities such that
$T_X=V_1 \oplus V_2$. If $K_X$ is not pseudoeffective there exists an algebraically integrable
subsheaf $\sF \subset T_X$ such that $\sF \subset V_i$ for $i=1$ or $2$.
\end{proposition}

Note that we do not claim that such a subsheaf exists for $i=1$ and $2$.

\begin{proof}
As in the proof of Theorem \ref{thm:main} we see that we can assume without loss of generality
that $X$ is $\Q$-factorial.
Since $K_X$ is not pseudoeffective we can run by \cite[Cor.1.3.3]{BCHM10} a terminating $K_X$-MMP
$X \dashrightarrow X'$ 
that ends with a Mori fibre space
$$
\holom{f}{X'}{Y}
$$
such that $-K_X$ is $f$-ample.
By Remark \ref{remark:MMP-induced splitting} the splitting $T_X=V_1 \oplus V_2$
induces a splitting $T_{X'}=V_1' \oplus V_2'$. In particular $-K_{X'} = \det V_1' + \det V_2'$,
so if $\R^+ \Gamma$ is the extremal ray contracted by $f$ we have $\det V_i' \cdot \Gamma>0$
for $i=1$ or $2$. Since the contraction is elementary this implies that $\det V_i'$ is $f$-ample. 
Thus we can apply
Proposition \ref{prop:inclusion-rel-tangent} to obtain $T_{X'/Y} \subset V_i'$.
\end{proof}

We conclude with an ``equivariant'' version of Beauville's example \cite[Sect.2.2]{Bea98}:

\begin{example} \label{example-not-integrable}
Let $\zeta$ be a third root of unity, and let $E:= \C/(\Z \oplus \Z \zeta)$, so the action of
$\zeta$ on the elliptic curve $E$ gives an automorphism $i_E$ of order three. The automorphism
$i_S := (i_E, i_E)$ acts on $S:= E \times E$. We denote by $\frac{\partial}{\partial z_i}$
a basis of $H^0(S, T_S)$ and note that $(i_S)_*   \frac{\partial}{\partial z_i}
= \zeta \frac{\partial}{\partial z_i}$. Thus if 
$Y := S/\langle i_S \rangle$ the quotient, the vector fields $\frac{\partial}{\partial z_i}$
do not descend to $Y$ but we have a decomposition
$$
T_Y \simeq M \oplus M
$$
where $M$ is a reflexive sheaf such that $M^{[3]} \simeq \sO_Y$. The surface
$Y$ has non-canonical klt singularities of type $\frac{1}{3}(1,1)$.
Since $K_Y \equiv 0$ the the minimal resolution $Y' \rightarrow Y$ is uniruled.
Since $h^1(Y', \sO_{Y'})=h^1(Y, \sO_Y)=0$ we deduce that $Y'$ and hence $Y$ is rationally connected.

The matrix 
$$
\begin{pmatrix}
\zeta & 0 & 0 \\
0 & \zeta^2 & 0 \\
0 & 0 & 1
\end{pmatrix}
$$
defines an automorphism $i_{\PP^2}$ of order three on $\PP^2$. Let $x_1, x_2, x_3$ be linear coordinates on $\C^3$ and set
$$
s_1 := x_3 \frac{\partial}{\partial x_1} + x_1 \frac{\partial}{\partial x_2}, \qquad
s_2 := x_2 \frac{\partial}{\partial x_3}.
$$
An elementary computation shows that $(i_{\PP^2})_* s_i = \zeta s_i$
and the Lie bracket $[s_1, s_2]$ does not vanish.

Consider now $X' := S \times \PP^2$. The vector fields 
$v_i := \frac{\partial}{\partial z_i} + s_i$
are linearly independent in every point, so they define a trivial subbundle $V \subset T_{X'}$
and $T_{X'} = V \oplus T_{X'/S}$.
One checks easily that $[v_1, v_2] = [s_1, s_2]$ is not a section of $V$ so
the distribution $V \subset T_{X'}$ is not integrable. Note also that
$$
(i_S \times i_{\PP^2})_* v_i = \zeta v_i,
$$
so the subbundle $V = \langle v_1, v_2 \rangle \subset T_{X'}$ is invariant under this automorphism.

Set $X := X'/\langle i_S \times i_{\PP^2} \rangle$, then $X$ is normal $\Q$-factorial variety with klt singularities. The fixed locus of the automorphism
$i_S \times i_{\PP^2}$ has codimension two, so the quotient map $\holom{\eta}{X'}{X}$ is 
quasi-\'etale and $T_{X'} \simeq \eta^{[*]} T_X$. 
Since $V \subset T_{X'}$ is invariant under the group action it descends to a distribution
$\sF \subset T_X$. The group action is diagonal, so the projection $X' \rightarrow S$ 
induces a fibration $\holom{f}{X}{Y}$
such that we have a splitting
$$
T_X = \sF \oplus T_{X/Y}.
$$
The general $f$-fibre is isomorphic to $\PP^2$. Since $Y$ is rationally connected we know by Tsen's theorem that $X$ is rationally connected. Since $\eta^{[*]} \sF = V \subset T_{X'}$ is not integrable, neither is $\sF \subset T_X$. 
\end{example}

\end{document}